\documentclass[11pt]{amsart}

\newcommand{\Complex}{\mathbb C}

\def\IC{{\mathbb{C}}}

\def\cV{{\mathcal V}}
\def\cN{{\mathcal N}}
\def\x{{\bf x}}
\def\y{{\bf y}}

\def\glim{{\rm glim}}

\newtheorem{them}{Theorem}[section]
\newtheorem{lem}[them]{Lemma}
\newtheorem{cor}[them]{Corollary}

\theoremstyle{definition}

\newtheorem{Exam}[them]{Example}

\numberwithin{equation}{section}

\providecommand{\AMS}{$\mathcal{A}$\kern-.1667em%
\lower.25em\hbox{$\mathcal{M}$}\kern-.125em$\mathcal{S}$}

\begin{document}

\title[product of operators]{Spectral radius, numerical radius, and the product of operators}

\author[R. Alizadeh ]{Rahim Alizadeh}
\address{(Alizadeh) Department of Mathematics, Shahed University,
P.O. Box 18151-159, Tehran, Iran}
\email{alizadeh@shahed.ac.ir}

\author[ M.B. Asadi]{Mohammad B. Asadi}
\address{(Asadi) School of Mathematics, Statistics and Computer Science,
 College of Science, University of Tehran, Tehran,
  Iran, and School of Mathematics, Institute for Research in Fundamental Sciences (IPM),
Tehran 19395-5746, Iran}
\email{mb.asadi@khayam.ut.ac.ir}

\author[C.-M. Cheng ]{Che-Man Cheng*}
\address{(Cheng and Hong) Department of Mathematics, University of Macau, Macao, China}
\email{(Cheng) fstcmc@umac.mo}
\email{(Hong) 675642073@qq.com}

\author[W. Hong ]{Wanli Hong}

\author[C.-K. Li]{Chi-Kwong Li}
\address{(Li) Department of Mathematics, College of William and Mary, Williamsburg, VA 23187,
USA}
\email{ckli@math.wm.edu}

\thanks{*Corresponding author}

\begin{abstract}
Let $\sigma(A)$, $\rho(A)$ and $r(A)$ denote the spectrum, spectral radius and numerical
radius of a bounded linear operator $A$ on a Hilbert space $H$,
respectively. We show that a linear operator $A$ satisfying
$$\rho(AB)\le r(A)r(B) \quad\text{ for all bounded linear operators } B$$
if and only if there is a unique $\mu \in \sigma (A)$ satisfying
$|\mu| = \rho(A)$ and $A = \frac{\mu(I + L)}{2}$ for a contraction
$L$ with $1\in\sigma(L)$. One can get the same conclusion on $A$ if $\rho(AB) \le
r(A)r(B)$ for all rank one operators $B$. If $H$ is of finite
dimension, we can further decompose $L$ as a direct sum of $C
\oplus 0$ under a suitable choice of orthonormal basis so that
$Re(C^{-1}x,x) \ge 1$ for all unit vector $x$.
\end{abstract}
\maketitle

\section{Introduction}

Let $B(H)$ be the algebra of bounded linear operators
$A$ acting on the Hilbert space $H$ with the inner product $(x,y)$.
We assume that $H$ has dimension at least 2 to avoid trivial consideration.
If $H$ is of dimension $n < \infty$, we identify $B(H)$ as the set $M_n$ of
$n\times n$ complex matrices acting on
$\IC^n$ equipped with the usual inner product $(x,y) = y^*x$.

Let $A \in B(H)$. Denote by
$\sigma(A)$ the spectrum of $A$ and
$\rho(A)$ the spectral radius of $A$. Furthermore, let
$$W(A)= \{(Ax,x): x \in H, (x,x) = 1\}\  \hbox{ and } \
r(A) = \sup\{|\mu|: \mu \in W(A)\}$$
be the numerical range and numerical radius of $A$, respectively.
The numerical range $W(A)$ is a bounded convex set in $\mathbb{C}$,
and  the numerical radius $r(A)$ is a norm on $B(H)$ satisfying
$$\rho(A) \le r(A) \le \|A\| \le 2r(A).$$
The numerical range and numerical radius are useful concepts in
studying linear operators. One may see \cite[Chapter 1]{Hor} or
\cite{Gus} for some basic background.

It is known (see \cite[Corollary 1.7.7]{Hor} and also \cite{Li}) that if $A$ is
a scalar multiple of a positive semidefinite operator, then
\begin{equation}\label{spectrum}
\sigma(AB)\subseteq W(A) W(B)\quad \mbox{ for all } B\in B(H).
\end{equation}
In \cite{Ali, Chg}, it is shown that the converse of this result is also true
if $H$ is of finite dimension. However, it may not be the case for the infinite
dimensional case; see \cite{Li}.
In this note, we consider $A \in B(H)$ satisfying the weaker condition that
\begin{equation}\label{nr}
\rho(AB)\le r(A)r(B) \quad\text{ for all } B\in B(H).
\end{equation}
It turns out that there is no difference in the finite and infinite dimensional case,
except that one can give some more detailed description of the structure of $A$
in the finite dimensional case. Here is our main theorem.

\begin{them}\label{necessary}
Let $A\in B(H)$ be nonzero. The following are equivalent.
\begin{itemize}
\item[{\rm(a)}] $ \rho(AB) \leq r(A)r(B)$ for all $B \in B(H)$.
\item[{\rm(b)}] $ \rho(AB) \leq r(A)r(B)$ for all rank one $B \in B(H)$.
\item[{\rm(c)}] There is a unique $\mu \in \sigma(A)$ attaining
$|\mu| = \rho(A)$, that satisfies
$\|A/\mu - I/2\| \le 1/2$, i.e.,
$A = \mu (I+L)/2$ for a contraction $L$ with $1 \in \sigma(L)$.
\end{itemize}
In case $H$ has dimension $n < \infty$,
conditions {\rm (a) -- (c)} are equivalent to the
following.
\begin{itemize}
\item[{\rm(d)}] There is a unique $\mu \in \sigma(A)$ satisfying $|\mu| = \rho(A)$ such that
the matrix $A$ is unitarily similar to
$$
\mu(I_p\oplus 0_q \oplus C)
$$
where $1 \le p \le n$, $0 \le q \le n-p$, and $C$ is invertible with
$\|C-I/2\| \le 1/2$, equivalently,
$$W(C^{-1}) \subseteq \{z :  \mbox{Re}(z)\ge 1\}.$$
\end{itemize}
\end{them}

Note that in condition (d), if $p+q = n$ then $C$ is vacuous;
if $C$ is vacuous or if $C$ is positive definite, then
$A$ is a multiple of a positive semidefinite operator so that
$\sigma(AB) \subseteq W(A)W(B)$ for all $B \in M_n$.
On the other hand, it is easy to find a non-normal matrix $A$ that satisfies (\ref{nr}).

\begin{Exam}\label{example}
Let
\begin{eqnarray}
A= \frac{1}{2}I_3 + \frac{1}{2} \left(
  \begin{array}{ccc}
    1 & 0 & 0 \nonumber \\
    0 & 0 & 1 \\
    0 & 0 & 0  \nonumber
  \end{array}
\right).
\end{eqnarray}
Then $A$ is not normal, and $A = \frac{1}{2}(I+L)$, where $L$ is a
contraction with $1 \in \sigma(L)$. By Theorem \ref{necessary},
$\rho(AB) \leq r(A) r(B)$ holds for all $B \in M_n$.
\end{Exam}

Note also that in the infinite dimensional case, condition (d) is
not equivalent to the conditions (a) -- (c). For example, if $H$ is a
separable Hilbert space with an orthonormal basis $\{e_1, e_2,
\dots\}$ and $A$ is the diagonal operator $A e_m = e_m/m$ for $m =
1, 2, \dots$. Then condition (c) holds, but the inverse of $A$ (and hence of $C$) is not bounded.
Also, if $A e_m = [m/(m+1)]e_m$ for $m = 1, 2, \dots$, then condition (c) holds with $\mu=1$.
However, $A$ is not unitarily similar to the form $[1]\oplus T$.

We further remark that the last equivalent condition in
condition (d) of Theorem \ref{necessary} seems to be more complicated
as it involved the inverse of $C$.
Nevertheless, we would like to include this condition because if
$A$ is of such a form, it follows from a known result,
\cite[Theorem 1.7.6]{Hor}, and Lemma \ref{invertible} below that condition (a) in Theorem 1.1 is true.
In that result, again, the inverse of a matrix is involved.

\section{Proof of Theorem \ref{necessary} and auxiliary results}

We begin with some lemmas. The first three can be verified easily.

\begin{lem}\label{invertible}
Let $A=A'\oplus 0\in M_n$ where $A'\in M_m$ is invertible.
Then
\begin{equation*}\label{AB}
\rho(AB)\le r(A)r(B) \quad\text{for all } B\in M_n
\end{equation*}
if and only if
\begin{equation*}\label{A'B'}
\rho(A'B')\le r(A')r(B') \quad\text{for all } B'\in M_m.
\end{equation*}
\end{lem}

\begin{lem}\label{a11}
If $A=(a_{ij})\in M_n$ and
  $|a_{jj}|=||A||$, then $a_{ij}=a_{ji}=0$ for all $i \neq j$.
\end{lem}

We shall frequently use the following lemma without explicitly mentioning it.

\begin{lem}\label{usi}
Suppose $A,\tilde{A}\in B(H)$ are unitarily similar.
Then
$$\rho(AB) \leq r(A)r(B) \quad\mbox{for all } B \in B(H)$$
if and only if
$$\rho(\tilde{A}B) \leq r(\tilde{A})r(B)\quad\mbox{ for all } B \in B(H).$$
\end{lem}

\begin{lem}\label{r1nr}
Suppose $A\in B(H)$ has rank at most one. Then
$$
r(A)=\frac12\left(\sqrt{{\rm tr}(A^*A)}+|{\rm tr}(A)|\right).
$$
\end{lem}

\begin{proof}
Suppose $A\in B(H)$ has rank at most one.
Then, $A$ is unitarily similar to $X\oplus 0$ where $X\in M_2$ has eigenvalues 0
and $\text{tr}(A)$.
By the elliptical range theorem (i.e., \cite[1.3.6]{Hor}), the result can be deduced readily.
\end{proof}

In the finite dimensional case,
$W(A)=\overline{W(A)}$ is compact, every element in $\sigma(A)$
is an eigenvalue, and there is a  unit vector $x\in H$ attaining
$\|Ax\| = \|A\|$.  However, it might not be the case in the infinite dimensional case.
Nevertheless,  we can use the Berberian construction (see \cite{B62}) to overcome
this obstacle in our proof.
We will obtain another lemma, which is
similar to \cite[Lemma 2.4]{Li}.
We include the details in the following for completeness and easy reference.

We identify the space $\ell_\infty$ of bounded scalar sequences with  the $C^*$-algebra
$C(\beta\mathbb{N})$ of continuous functions on the Stone-Cech compactification
$\beta \mathbb N$ of $\mathbb N$.  Here,
a bounded sequence $\lambda=(\lambda_n)$ in $\ell_\infty$ corresponds
to a function $\hat{\lambda}$ in  $C(\beta\mathbb{N})$
with $\hat{\lambda}(n)=\lambda_n$ for all $n=1,2,\ldots$.
Take any point $\xi$ from $\beta\mathbb{N}\setminus \mathbb{N}$.
The point evaluation $\lambda\mapsto \hat{\lambda}(\xi)$ of $\ell_\infty$ gives a
nonzero multiplicative
generalized Banach limit, denoted by $\glim$, that  satisfies the following conditions.
For any bounded sequences $(a_n)$ and $(b_n)$ in $\ell_\infty$ and scalar $\gamma$,
we have
\begin{enumerate}
    \item $\glim (a_n+b_n) = \glim (a_n) + \glim (b_n)$.
    \item $\glim (\gamma a_n) = \gamma \glim (a_n)$.
    \item $\glim (a_n) = \lim a_n$ whenever $\lim a_n$ exists.
    \item $\glim (a_n) \ge 0$ whenever $a_n \ge 0$ for all $n$.
    \item $\glim (a_nb_n) = \glim (a_n)\glim (b_n)$.
\end{enumerate}
Note that (iv) implies $\glim(a_n)$ is real if  $\{a_n\}$ is real.
It follows that
$\glim(\overline a_n) = \overline{\glim(a_n)}$ and
$\glim (a_n \overline a_n) = \glim |a_n| \glim |\overline{a_n}|= \glim(a_n)^2$.
(Of course, this also follows from the
identification of the space $\ell_\infty$ of bounded scalar sequences with
the $C^*$-algebra
$C(\beta\mathbb{N})$.)

Denote by $\cV$ the set of all bounded
sequences $ \{x_n\}$ with $x_n \in  H$. Then $\cV$ is a vector
space relative to the definitions $\{x_n\}+\{y_n\} = \{x_n + y_n\}$
and $\gamma \{x_n\} = \{\gamma x_n\}$.  Let $\cN$
be the set of all   sequences $\{x_n\}$ such that $\glim (\langle x_n,x_n\rangle) = 0$.
Then $\cN$ is a
linear subspace of $\cV$. Denote by $\x $ the coset $\{x_n\}+\cN$. The
quotient vector space $\cV/\cN$ becomes an inner product space with
the inner product
$\langle \x ,\y \rangle  = \glim (\langle x_n,y_n\rangle)$.
Let $K$ be the completion of $\cV/\cN$. If
$x \in H$, then $ \{x\}$ denotes the  constant sequence defined by
$x$. Since $\langle\x,\y\rangle =  \langle x,y\rangle $ for $\x  =
\{x\}+\cN$ and $\y  = \{y\}+\cN$, the mapping $x \mapsto \x $ is an isometric
linear map of $H$ onto a closed subspace  of $K $
and $K$ is an
extension of $H$. For an operator $T \in  {\mathcal B}(H)$, define
$$T_0(\{x_n\}+\cN) = \{Tx_n\} + \cN.$$
We can extend $T_0$ on $K$, which will be denoted by $T_0$ also.  The
mapping $\phi: {\mathcal B}(H)\to  {\mathcal B}(K)$ given by $\phi(T)= T_0$  is
a unital  isometric
$*$-representation with $\sigma(T)=\sigma(T_0)$.
Moreover, the approximate eigenvalues of $T$ (and also $T_0$)
will become eigenvalues of $T_0$; see \cite{B62}.

\begin{lem} \label{2.5}
Let $\tilde A \in B(K)$ be the extension of $A \in B(H)$ in the Berberian construction.
Suppose $\rho(AB) \le r(A)r(B)$ for all rank one $B \in B(H)$.
Then
$$
\rho(\tilde A B') \le r(\tilde A)r(B') \text{ for all rank one } B' \in B(K).
$$
\end{lem}

\begin{proof}
Suppose $\tilde A \in B(K)$ be the extension of $A \in B(H)$ in the Berberian construction.
Then $r(A) = r(\tilde A)$.
Let $B' = \x \otimes \y$ for two unit vectors  $\x, \y$ in $K$
associated with the sequences $\{x_n\}, \{y_n\}$ of
unit vectors in $H$.
Then
$$r(B')
= (1 + |(\x,\y)|)/2 = (1+\glim |(x_n,y_n)|)/2$$
and
$$\rho(\tilde A B') = \rho(\tilde A,\x\otimes \y) =
|(A\x,\y)| = \glim |(A x_n,y_n)|.$$
By our assumption on $A$,
$$
|(Ax_n,y_n)| = \rho(A(x_n\otimes y_n))
\le r(A) r(x_n\otimes y_n) = r(\tilde A)(1+ |(x_n,y_n)|)/2$$
for all positive integer $n$.
It follows that
$$\rho(\tilde A B') =\glim |(A x_n,y_n)| \le
r(\tilde A) \glim (1+|(x_n,y_n)|)/2
= r(\tilde A)r(B').$$
\vskip -.3in
\end{proof}

\bigskip\noindent
{\bf Proof of Theorem \ref{necessary}}

The implication ``(a) $\Rightarrow$ (b)'' of Theorem \ref{necessary}
is clear. We will establish the implications ``(b) $\Rightarrow$
(c)'' and ``(c) $\Rightarrow$ (a)''.

Suppose (b) holds. Without loss of generality we can suppose that
$\|A\|=1$. For arbitrary unit vectors $x,y \in H$, we consider the
rank one operators $B_{x,y} =x \otimes y$. Then
$$
|(Ax,y)| = \rho(AB_{x,y}) \leq r(A) r(B_{x,y}) \leq r(A).
$$
Taking supremum over $x$ and $y$, we conclude that $r(A)=\|A\|$.
Hence $\rho(A)=\|A\|$ and $A$ is radial, see \cite[Theorem 1.3-2]{Gus}.
Let $\tilde A \in B(K)$
be the extension of $A$ in the Berberian construction. Then by Lemma \ref{2.5},
$$\rho(\tilde A B) \le  r(\tilde A) r(B) \quad \hbox{for all rank
one} \ B \in B(K).$$ By the above discussion $\tilde A$ is a radial
operator. Hence there is a  $ \mu \in \sigma(\tilde A)$ such that
$|\mu|=\|\tilde A\| = 1$. Therefore $ \mu \in
\partial \sigma(\tilde A)$. On the other hand, we know that
every boundary point of the spectrum is an approximate eigenvalue.
Therefore, $\mu$ is an eigenvalue of $\tilde A$  and $\tilde A$ has
a decomposition $[\mu] \oplus T$ on $K= \langle x \rangle  \oplus \ K_1$. Without
loss of generality we can assume $\mu=1$, i.e., $\tilde A=[1]\oplus
T \in B(\IC \oplus K_1)$. Suppose that $B=u\otimes v$, where
$$u=\begin{pmatrix} \sqrt{1-t}\cr \sqrt{t}x\cr\end{pmatrix}, \quad
v=\begin{pmatrix} \sqrt{1-t}\cr \sqrt{t}y\cr \end{pmatrix}$$ with
$0\le t\le 1$, $x,y\in K_1$ are unit vectors. Because $B$ has rank 1
and $u$, $v$ are unit vectors in $K$, by Lemma \ref{r1nr}
$r(B)=\frac12(1+|\mbox{tr}(B)|) = \frac{1}{2}(1+|(u,v)|)$. Thus
$$
|(\tilde Au,v)|=|\mbox{tr}(\tilde AB)|=\rho(\tilde AB)\le r(\tilde
A)r(B)=\frac12(1+|(u,v)|),
$$
i.e.,
\begin{equation}\label{one}
|1-t+t(Tx,y)| \le
\frac12(1+|1-t+t(x,y)|).
\end{equation}
When $t$ is small enough, by Taylor series expansion we have
\begin{eqnarray*}
& &|1-t+t(Tx,y)| \\
&=& \sqrt{1+2t(\mbox{Re}(Tx,y)-1)+t^2[(\mbox{Re}(Tx,y)-1)^2+(\mbox{Im}(Tx,y))^2]}\\
             &=& 1+t(\mbox{Re}(Tx,y)-1)+O(t^2),
\end{eqnarray*}
and similarly
$$
|1-t+t(x,y)|=1+t(\mbox{Re}(x,y)-1)+O(t^2).
$$
Hence, by (\ref{one}),
$$
1+t(\mbox{Re}(Tx,y)-1)\le \frac12[1+1+t(\mbox{Re}(x,y)-1)+O(t^2)],
$$
and thus we get
$$
t\mbox{Re}\left((T-I/2)x,y\right) \le t/2+O(t^2),
$$
Consequently, we have
$$
\mbox{Re}\left((T-I/2)x,y)\right) \le \frac12.
$$
Because  $x,y\in K_1$ are arbitrary unit vectors, we see that
$\|T-I/2\| \le 1/2$. Hence $\|\tilde A-I/2\| \le 1/2$ which
implies $\|A-I/2\| \le 1/2$. Now taking $L = 2A -I$, we get the
desired result.

Next, we establish the implication ``(c) $\Rightarrow$ (a)''.
Suppose (c) holds with $\mu =1$, and (a) is not true. The
assumption implies $r(A)=1$ and so there is $B \in B(H)$ with $
r(B) = 1$ and $\rho(BA) = \rho(AB) > 1$. Therefore, there is a
$\lambda$ in the approximate spectrum of $BA$ such that
$|\lambda |= \rho(BA)$. We may apply the Berberian construction to
extend $H$ to $K$ so that $A, B$ are extended to $\tilde A,
\tilde B \in B(K)$  and $\lambda$ is an eigenvalue of
$\tilde B\tilde A$. For notational simplicity, we assume that $H = K$ and
$(A,B) = (\tilde A, \tilde B)$. Let $x \in H$ be such that  $BAx
= \lambda x$ with $|\lambda| > 1$. Let $Ax = a_{11}x + a_{21}y$
such that $\{x,y\}$ is an orthonormal set, and let $U$ be unitary
with $x, y$ as the first two columns. Then
$$U^*AU = \begin{pmatrix} a_{11} & * & * \cr a_{21} & * & * \cr 0 & * & * \cr\end{pmatrix},
\quad U^*BU = \begin{pmatrix} b_{11} & b_{12} & * \cr b_{21} &
b_{22} & * \cr * & * & * \cr\end{pmatrix}$$ such that
$$U^*BAU = \begin{pmatrix} \lambda & * & * \cr 0 & * & * \cr 0 & * & * \cr\end{pmatrix}.$$
Because $\|A - I/2\|\leq 1/2$, the first column of $U^*AU - I/2$
has norm at most $1/2$, i.e.,
$$|a_{11}-1/2|^2 + |a_{21}|^2 \le 1/4.$$
Thus, there is $\xi \ge 1$ such that
$$|\xi a_{11} - 1/2|^2 + |\xi a_{21}|^2 = 1/4,$$
i.e., $2(\xi a_{11} -1/2, \xi a_{21})^t$ is a unit vector in
$\Complex^2$. So, $\xi a_{11} - 1/2$ lies inside the circular
disk centered at the origin with radius 1/2, and is contained in
a line segment joining $-1/2$ to $e^{it}/2$ for some $t \in [0,
2\pi)$.  Hence, $2(\xi a_{11} - 1/2) \in W(V)$ for any unitary
matrix $V\in M_2$ with eigenvalues $-1, e^{it}$. In particular,
we can construct a unitary matrix $V \in M_2$  with eigenvalues
$-1, e^{it}$ such that the $(1,1)$ entry equals $2\xi a_{11}-1$.
Furthermore, we may replace $V$ by $\mbox{diag}(1,
e^{ir})V\mbox{diag}(1, e^{-ir})$, where $r=Arg( \xi a_{21})$, and
assume that the $(2,1)$ entry of $V$ is $2\xi a_{21}$. Thus, the
first column of $V$ equals $(2\xi a_{11} - 1, 2\xi a_{21})^t$. Set
$ \hat A = (I_2+V)/2 \oplus 0$ and
$$\hat B = \begin{pmatrix} b_{11} & b_{12}  \cr b_{21} & b_{22}\cr
\end{pmatrix} \oplus 0.$$
Then $r(\hat B) \le r(B) \le 1$, $W(\hat A)$ is a line segment
joining $0$ and $(1 + e^{it})/2$, and
$$\hat B \hat A = \begin{pmatrix}\xi \lambda & * \cr
0 & *  \cr\end{pmatrix} \oplus 0.$$ Since $\hat A$ is a scalar
multiple of a positive operator,  we have
$$\xi\lambda \in \sigma(\hat B \hat A) \subseteq W(\hat B) W(\hat A),$$
and thus,
$$1 < \xi|\lambda| \leq \rho(\hat A \hat B)
\le r(\hat A)r( \hat B) \le 1,$$ which is a contradiction.

Now suppose $H$ has dimension $n < \infty$.
We may assume that $A \in M_n$, and
we will prove that (c) $\iff$ (d).

The implication ``(d) $\Rightarrow$ (c)''
follows readily with $\|C-I_r/2\|\le 1/2$. On the other hand, for invertible $C\in M_{r}$,
\begin{eqnarray}
\left\|C-\frac12I_r\right\| \leq \frac12 &\Leftrightarrow& \|I_r-2C\|^2 \leq 1 \nonumber \\
&\Leftrightarrow& I_r \geq (I_r-2C)(I_r-2C^*) \nonumber \\
&\Leftrightarrow&  I_r \geq I_r-2C-2C^*+4CC^* \nonumber \\
&\Leftrightarrow&  I_r \geq I_r - 2C({C^*}^{-1}+C^{-1}-2I_r)C^* \nonumber\\
&\Leftrightarrow& \text{Re}(C^{-1}-I_r) \geq 0 \nonumber\\
&\Leftrightarrow& W(C^{-1}) \subseteq \{z ~ : ~\mbox{Re}(z)\ge 1\}. \nonumber
\end{eqnarray}

Suppose (c) holds. Then $A \in M_n$ satisfies $\rho(A) = r(A) =
\|A\| = |\mu|$, where $\mu \in \overline{W(A)} = W(A)$. For
notational simplicity, assume that $\mu=1$. Then $A$ is unitarily
similar to $I_p \oplus A_1$, where $A_1$ is in upper triangular
form such that $1\notin \sigma(A_1)$. If $A_1$ is invertible, take
$C=A_1$. Suppose $A_1$ is singular. Then, as $A_1 = (\alpha_{ij})$
is in triangular form, and some of its diagonal elements are zero.
If $\alpha_{ii}=0$ then the $(i,i)$ diagonal element of
$A_1-\frac12I_{n-p}$ is $\frac12$. Thus, by Lemma \ref{a11}, we
have $\alpha_{i,j}=\alpha_{j,i}=0$, $j \ne i$. Consequently, $A_1$
is permutational similar to $0_q \oplus C$, where $C$ is
invertible. Because $\|A - I_{n}/2\| = 1/2$,
we see that $\|C -I_{n-p-q}/2\| \le 1/2$. By the previous argument,  we see that $
W(C^{-1}) \subseteq \{z ~ : ~\mbox{Re}(z)\ge 1\}$. \qed
\\

If $A$ is a normal matrix, we have the following corollary.
\begin{cor}  Let $A \in M_n$ be normal with eigenvalues $\{\lambda_1,\cdots,\lambda_n\}$ ordered so that $|\lambda_1| \geq
\cdots \geq |\lambda_n|$. Then the following expressions are
equivalent:

$(a)$ $ \rho(AB) \leq r(A)r(B)$, for all $B \in M_n$,

$(b)$ $|2 \lambda_j-\lambda_1| \leq |\lambda_1|$, for every
$j=2,\dots,n.$

\end{cor}

\medskip\noindent
{\bf Acknowledgment}

Asadi was partly supported by a grant from
IPM (No. 92470123).

Cheng was supported by
the research grant MYRG065(Y2-L2)-FST13-CCM from University of Macau.

Li is an affiliate member of the Institute for Quantum Computing, Waterloo,
an honorary professor of the University of Hong Kong and the Shanghai University.
His research is supported by USA NSF and HK RCG.

\end{document}